\documentclass[11pt,oneside]{article}
\usepackage[utf8]{inputenc}
\usepackage[english]{babel}
\usepackage{amsmath}
\usepackage{amsfonts}
\usepackage{amssymb}
\usepackage{amsthm}
\usepackage{tikz-cd} 
\usepackage{hyperref}
\usepackage{comment}
\usepackage{pgfplots}
\usetikzlibrary{calc}

\usepackage[left=3cm,right=3cm,top=3cm,bottom=3cm]{geometry}
\usepackage[title]{appendix}
\usepackage[]{algorithm2e}

\usepackage[normalem]{ulem}
\useunder{\uline}{\ul}{}
\theoremstyle{plain}
\newtheorem{teo}{}[section]
\newtheorem{prop}[teo]{Proposition}
\newtheorem{cor}[teo]{Corollary}

\newtheorem{lem}[teo]{Lemma}
\newtheorem{thm}[teo]{Theorem}
\theoremstyle{definition}
\newtheorem{ex}[teo]{Example}
\newtheorem{rem}[teo]{Remark}
\newtheorem{df}[teo]{Definition}
\usepackage{graphicx} 

\newcommand\blfootnote[1]{%
  \begingroup
  \renewcommand\thefootnote{}\footnote{#1}%
  \addtocounter{footnote}{-1}%
  \endgroup
}

\title{On the existence and properties of Alexandroff paratopological groups}
\author{Tayomara Borsich and Pedro J. Chocano}
\date{}

\begin{document}

\maketitle

\begin{abstract} 

We study groups endowed with Alexandroff topologies and show that no non-discrete Alexandroff topology can turn a group into a topological group. This settles negatively the basic existence problem for Alexandroff topological groups. Motivated by this obstruction, we turn to the broader setting of Alexandroff paratopological groups. We establish several fundamental properties of these spaces regarding to connections and compact-like properties among others. As applications, we address two addapted classical open questions within our framework concerning feebly bounded subsets in paratopological groups, proving that connected non-compact Alexandroff paratopological groups offer a positive solution both for products of feebly bounded sets and for the feebly boundedness of $B^2$ when $B$ is a feebly bounded subset. 

\end{abstract}

\blfootnote{This research is partially supported by Grant  PID2024-156663NB-I00 from Ministerio de Ciencia, Innovación y Universidades (Spain) and  2025/SOLCON-159637 from Rey Juan Carlos University}

\section{Introduction}

The theory of finite topological spaces, or more broadly Alexandroff spaces, has proven to be a powerful tool for addressing problems of various kinds in mathematics. For example, there has been recent interest in applying these spaces to the study of dynamical systems (see \cite{barmak2024conley} or \cite{ChocanoMoronRuizdelPortal2025} and the references therein). They also appear in classical theoretical problems, such as Quillen’s conjecture (see \cite[Chapter~8]{barmak2011algebraic}).

In \cite{chocanoPrieto2026negativedimension}, the authors construct a group whose elements encode simplicial complexes of both negative and positive dimension and which also carry a natural topological structure. This topological structure is, in fact, an Alexandroff topology. However, the group structure is not compatible with this topology, and therefore the resulting object is neither a topological group nor a paratopological group. This observation motivates the following question, which is the starting point of the present work: can a group admit a non-discrete Alexandroff topology that turns it into a topological group? And in the affirmative case, what properties should such a group satisfy?

In this paper we answer the question in the negative: the only connected Alexandroff topology that turns a group into a topological group is the discrete topology on the one-point space. Although this result can be derived from classical results on the theory of topological groups, we provide an original prove that exploits the Alexandroff topology structure. This motivates us to consider instead the weaker notion of paratopological groups equipped with an Alexandroff topology. We show, through explicit examples, that such spaces do exist. Moreover, we prove that imposing compactness forces the topology to be discrete once again.

Although we restrict our attention to non-compact, $T_0$ Alexandroff paratopological groups, these spaces already exhibit interesting and sometimes unexpected behavior. In particular, this family can be used both to characterize certain classical properties appearing in the literature and to provide counterexamples to natural questions. For instance, we will be able to address the following problem: given paratopological groups $G_1$ and $G_2$ and feebly bounded subsets $B_i \subseteq G_i$ for $i = 1,2$, is the product $B_1 \times B_2$ necessarily a feebly bounded subset of the product group $G_1 \times G_2$? 
In particular, we will show that for any family of non-compact Alexandroff paratopological groups $X_\alpha$ with $\alpha \in I$ and feebly bounded subsets $B_\alpha \subseteq X_\alpha$, the product $\prod_{\alpha \in I} B_\alpha$ is a feebly bounded subset of $\prod_{\alpha \in I} X_\alpha$.

Another illustrative example of the usefulness of these spaces concerns the following problem: if $B$ is a feebly bounded subset of a paratopological group $G$, must $B^2$ also be feebly bounded? We will show that within our setting, if $X$ is a non-compact Alexandroff paratopological group and $B \subseteq X$ is feebly bounded, then indeed $B^2$ is feebly bounded.

The organization of the paper is as follows. In Section~\ref{sec_intro}, we introduce the notation, terminology, and basic results required to keep the paper as self-contained as possible. Section~\ref{sec_existencia_primeras_propiedades} is devoted to establishing the triviality of Alexandroff topological groups, proving some elementary properties of Alexandroff paratopological groups, and providing examples to familiarize the reader with these spaces. In Section~\ref{sec_topological}, we study further topological properties and address some of the open questions mentioned above. Finally, in Section~\ref{sec_combinatorics}, we provide a topological classification of a special family of Alexandroff paratopological groups.

\section{Preliminaries}\label{sec_intro}

Let us recall the main definitions and concepts from the theory of Alexandroff spaces and paratopological groups. For a comprehensive introduction to paratopological groups, we refer the reader to \cite{tkachenko2014paratopological}. For an introduction to the theory of Alexandroff spaces, we recommend the standard references \cite{barmak2011algebraic, may1966finite}.
\begin{df}\label{def_topological group}
    A \emph{topological group} \(X\) is a topological space equipped with a group structure such that
\begin{enumerate}
    \item[(i)] the multiplication map \(m : X \times X \to X\), defined by \(m(x,y) = xy\), is continuous; and
    \item[(ii)] the inversion map \(\mathrm{Inv} : X \to X\), defined by \(\mathrm{Inv}(x) = x^{-1}\), is continuous.
\end{enumerate}

If condition (ii) is omitted, then \(X\) is called a \emph{paratopological group}.
\end{df}

\textbf{Notation}. Throughout, we assume that for any paratopological group \(X\), 
the symbol \(1 \in X\) denotes the identity element of the group. For each \(x \in X\), let 
\(L_x : X \to X\) (resp., \(R_x : X \to X\)) denote the map 
\(L_x(y) = m(x, y)\) (resp., \(R_x(y) = m(y, x)\)). 
Morphisms in this category are maps that are continuous and also group homomorphisms. 
We simply say that a map \(f : X \to Y\) between two paratopological groups is a 
continuous homomorphism if it satisfies these conditions. 
If there exists a continuous homomorphism \(f^{-1} : Y \to X\) such that 
\(f \circ f^{-1}\) and \(f^{-1} \circ f\) are the identity maps, 
then we say that \(X\) is isomorphic to \(Y\).

\begin{df}\label{df_alexandroff}
    A topological space $X$ is an \emph{Alexandroff space} if the arbitrary intersection of open sets is an open set.
\end{df}
One of the most important examples of Alexandroff spaces is that of finite topological spaces, since any topology on a finite set automatically satisfies the property of Definition \ref{df_alexandroff}.

Given an Alexandroff \(T_0\)-space \(X\) and a point \(x \in X\), define  \(U_x\) (resp., \(F_x\)) as the intersection of all open (resp., closed) sets containing \(x\). Set \(x \leq y\) if and only if \(U_x \subseteq U_y\). This relation defines a partial order on \(X\). Conversely, for any partially ordered set \((X, \leq)\), the collection of lower sets forms a basis for an Alexandroff \(T_0\) topology on \(X\). 
Moreover, a map \(f : X \to Y\) between Alexandroff spaces is continuous if and only if \(f : (X,\leq) \to (Y,\leq)\) is order-preserving. From this, one obtains one of the fundamental theorems in the theory (see \cite{alexandroff1937diskrete}):
\begin{thm}
     The category of Alexandroff \(T_0\)-spaces is isomorphic to the category of partially ordered sets.
\end{thm}

Henceforth, we freely identify a $T_0$ Alexandroff space with its poset and use topological and order-theoretic terminology interchangeably. Unless otherwise stated, all Alexandroff spaces are assumed to be $T_0$.

\begin{rem}
Let \(X\) be an Alexandroff space. Then
\[
U_x = \{\, y \in X \mid y \leq x \,\}, \qquad
F_x = \{\, y \in X \mid y \geq x \,\}.
\]
Moreover, for any subset \(S \subseteq X\), define
\[
U_S = \{\, y \in X \mid y \leq s \text{ for some } s \in S \,\}, \qquad
F_S = \{\, y \in X \mid y \geq s \text{ for some } s \in S \,\}.
\]
Similarly, we define
\[
C_x = U_x \cup F_x, \qquad C_S = U_S \cup F_S.
\]
\end{rem}

When defining the partial order on an Alexandroff space \(X\), one may also consider the following relation: \(x \leq_o y\) if and only if \(U_y \subseteq U_x\). This relation defines another partial order on \(X\), which we call the \emph{opposite order}. If one instead takes upper sets in a 
partially ordered set \((X,\leq)\), then one obtains an Alexandroff \(T_0\) topology on \(X\), called the \emph{opposite topology}. We denote by \(X^{op}\) 
the set \(X\) equipped with this opposite topology or, equivalently, with the opposite partial order. It is straightforward to see that \(U_x\) corresponds to 
\(F_x\) in \(X^{op}\), and vice versa.

Given an Alexandroff space \(X\) and a point \(x \in X\), let \(A_x\) denote a maximal antichain containing \(x\). The \emph{width} of \(X\), denoted \(\mathrm{width}(X)\), is defined as the cardinality of a maximal antichain in \(X\). The \emph{height} of \(X\), denoted \(\mathrm{ht}(X)\), is defined as one less 
than the length of a maximal chain in \(X\). The height of a point \(x\) is defined by \(\mathrm{ht}(x) = \mathrm{ht}(U_x)\). Furthermore, we write \(x \prec y\) in \(X\) to indicate that there is no element \(z \in X\) such 
that \(x < z < y\). Note that if \(f : X \to Y\) is a homeomorphism, then \(f\) preserves heights, the width, chains, and antichains. The \emph{Hasse 
diagram} of \(X\) is the directed graph whose vertices are the elements of \(X\), with a directed edge from \(x\) to \(y\) whenever \(x \prec y\).

Let us recall a basic property of Alexandroff spaces:

\begin{prop}
    Let $X$ be an Alexandroff space. Then $X$ is connected if and only if it is path-connected. Moreover, a path from $x$ to $y$ in $X$ is a finite sequence of points $x_0,\ldots,x_n \in X$ such that $x_0 = x$, $x_n = y$, and for each $i$ we have either $x_i < x_{i+1}$ or $x_i > x_{i+1}$.
\end{prop}

Now, we introduce a useful construction within this setting:
\begin{df}
Let $X$ and $Y$ be Alexandroff spaces. The join $X\circledast Y$ is the Alexandroff space that consists on the disjoint union of $X$ and $Y$, where we keep the given ordering in $X$ and $Y$ and we declare that $x\leq y$ for every $x\in X$ and $y\in Y$.
\end{df}

To conclude this section, we recall some basic results concerning homotopy in this setting. We begin by introducing notation. Given continuous maps \(f, g : X \to Y\) between Alexandroff spaces, we write \(f \geq g\) (resp., \(f \leq g\)) if \(f(x) \geq g(x)\) (resp., \(f(x) \leq g(x)\)) for every \(x \in X\).
\begin{thm}
    Let $f,g:X\rightarrow Y$ be continuous maps between Alexandroff spaces. If there exists a sequence $f_0,...,f_n:X\rightarrow Y$ of continuous maps such that $f_0=f,~ g=g_n$ and $f_{i+1}\leq f_i$ or $f_{i+1}\geq f_{i}$ with $i=0,...,n-1$, then $f$ is homotopic to $g$.
\end{thm}

\begin{df}
    Let $X$ be an Alexandroff space and $x\in X$. It is said that $x$ is a \emph{down} (resp., \emph{up}) \emph{beat point} if $U_x\setminus \{x\}$ (resp., $F_x\setminus \{ x\}$) contains a maximum (resp., minimum). 
\end{df}
\begin{thm}
    Let $X$ be an Alexandroff space. If $x\in X$ is a down or up beat point, then $X\setminus \{x \}$ is a strong deformation retract of $X$.
\end{thm}

\section{On the existence of Alexandroff topological groups and elemental properties}\label{sec_existencia_primeras_propiedades}

We begin this section by proving a sort of results that will play a key role in the developments of this section and subsequent sections.

\begin{lem}\label{lem_aux_x_en_u_1_entonces_inverso_f_1} Let $X$ be an Alexandroff paratopological group. If $x\in F_1$, then $x^{-1}\in U_1$.
\end{lem}
\begin{proof}
    The condition $x\in F_1$ is equivalent to $x\geq 1$. Since $X$ is a paratopological group, the map $R_{x^{-1}}$ is continuous. Therefore, $R_{x^{-1}}(x)=xx^{-1}=1\geq  x^{-1}=1x^{-1}=R_{x^{-1}}(1)$, which gives $x^{-1}\in U_1$.
\end{proof}
From Lemma \ref{lem_aux_x_en_u_1_entonces_inverso_f_1}, it is easy to deduce the following result.
\begin{lem}\label{lem_opuesto_u_1} Let $X$ be an Alexandroff paratopological group. Then $U_1^{op}=F_1$. 
\end{lem}

Indeed, we have a generalization for open sets:
\begin{lem}\label{lem_inversas_cerrados_en_abiertos}
    Let $U$ be an open set of a non-compact Alexandroff paratopological group $X$. Then $U^{-1}$ is a closed set in $X$.
\end{lem}
\begin{proof}
    Let $x\in U$. We verify that for every $y>x^{-1}$, one has that $y\in U^{-1}$, i.e., $U^{-1}$ is an upper set. Recall that an upper set is a closed set. Hence, by proving the last assertion we conclude. If $y>x^{-1}$, then $x>y^{-1}$. Since $U$ is an open set, we obtain that $y^{-1}\in U$ and, consequently, $y=(y^{-1})^{-1}\in U^{-1}$.
\end{proof}

\begin{thm}\label{thm_no_subgrupos_cerrados_abiertos}
    Let $X$ be a connected non-compact Alexandroff paratopological group, then $X$ does not have proper open (closed) subgroups.
\end{thm}
\begin{proof}
   This follows immediately from Lemma~\ref{lem_inversas_cerrados_en_abiertos}.  If \(U\) is a proper open subgroup of \(X\), then \(U^{-1} = U\), which implies that \(U\) is also closed. The connectedness hypothesis then yields the result.
\end{proof}

Moreover, from Lemma \ref{lem_inversas_cerrados_en_abiertos}, we obtain the following result:

\begin{thm}
    Let $X$ be a non-compact Alexandroff paratopological group. Then the conjugate topology coincides with the opposite topology.
\end{thm}

For any given paratopological group \((G,\tau)\), it is well known that \(G^* = (G,\, \tau \vee \tau^{-1})\) is a topological group. This group is called the topological group \emph{associated} to the paratopological group 
\(G\) (see \cite[Section~1]{tkachenko2014paratopological} for more details and notation). In our setting, the topology of the topological group associated to a non-compact Alexandroff paratopological group is trivially the discrete 
topology. This is consistent with the fact that if \(G\) satisfies the \(T_0\) separation axiom, then \(G^*\) is Hausdorff. Moreover, in general, an Alexandroff topology and its opposite need not coincide.

We now state the main theorems of this section. The following results justify the hypotheses introduced earlier, namely non‑compactness and the assumption that we are working with paratopological groups.
\begin{thm}\label{thm_no_existen_grupos_topologicos_alexandroff} The only connected Alexandroff topological group is the trivial topological  one-point space.
\end{thm}
\begin{proof}
    Let $X$ be a connected Alexandroff topological group. We argue by contradiction. Suppose  $x\in X$ such that $x\neq 1$. Without loss of generality, and by the connectedness property, we may assume that $x>1$. By Lemma \ref{lem_aux_x_en_u_1_entonces_inverso_f_1}, $x^{-1}<1$. This leads to a contradiction with (ii) in Definition \ref{def_topological group} because $x^{-1}\not > 1^{-1}=1$.
\end{proof}
\begin{thm}\label{thm_no_existen_finitos_paratopologicos_grupos}
Let $X$ be a compact Alexandroff paratopological group. 
\begin{enumerate}
    \item If $X$ is connected, then $X$ is the one-point space.
    \item If $X$ is not connected, then $X$ has the discrete topology.
\end{enumerate}

\end{thm}
\begin{proof}
    1.  By compactness, we may assume that there are a finite number of maximal points in $X$, otherwise by considering an infinite chain $x_1<x_2<...$ and a covering containing $U_{x_i}$ we derive a contradiction. By connectedness, assume $1<x$, where $x$ is a maximal point. Then $x=L_x(1)<L_x(x)$, which gives a contradiction with the maximality of $x$.


    2. We argue by contradiction. Assume that $X$ does not have the discrete topology. Note that by repeating the arguments from 1, we can conclude that the connected component that contains the point $1$, denoted by $X^1$, consists only on this point. On the other hand, since $X$ does not have the discrete topology, there exists a connected component, denoted by $X'$, such that $x,y\in X'$ with $x<y$. Note that $L_x$ is a homeomorphism whose inverse is $L_{x^{-1}}$. By the continuity of $L_{x^{-1}}$, we deduce that $1=x^{-1}x\leq x^{-1}y$. However, this relations holds if and only if $x^{-1}y=1$ since $X^1=\{1 \}$, and this gives a contradiction with the fact that $L_{x^{-1}}$ is a homeomorphism.
\end{proof}

Theorems \ref{thm_no_existen_grupos_topologicos_alexandroff} and \ref{thm_no_existen_finitos_paratopologicos_grupos} give that within this setting the only non-trivial spaces are the non-compact paratopological Alexandroff spaces. To conclude this section, we present three examples of such spaces. The first two examples are abelian, while the third is not.

\begin{ex}\label{ex_reales_orden_usual}
    Let us consider the set of the real numbers $\mathbb{R}$ or integer numbers $\mathbb{Z}$ with the usual order and the sum. This is a non-trivial non-compact Alexandroff paratopological group.
\end{ex}
\begin{ex}\label{ex_loewner}
Let $\mathcal{M}_n$ the set of square matrices of order $n$. Consider the Loewner order on this set, i.e., $A\leq B$ if and only if $B-A$ is positive semi-definite. This defines an Alexandroff $T_0$ topology on  $\mathcal{M}_n$. Moreover, the usual sum of matrices is compatible with this partial order. Hence, $\mathcal{M}_n$ is a non-compact abelian paratopological Alexandroff group. Note also that $U_{\mathbf{0}}$ consists on the set of every positive semi-definite matrices, where $\mathbf{0}$ denotes the matrix of zeros.

Let us consider the real numbers $\mathbb{R}$ as in Example \ref{ex_reales_orden_usual}. It is easy to verify that the map $\textnormal{tr}:\mathcal{M}_n\rightarrow\mathbb{R}$ defined by the usual trace of matrices is a continuous homomorphism between non-compact Alexandroff paratopological groups. 
\end{ex}

    

\begin{ex}\label{ex_matrices_invertibles}
    Let $GL_n(\mathbb{R})$ be the group of invertible matrices with real coefficients where the group operation is the standard product of matrices. Set $A\leq B$ if and only if $\log(|\det(A)|)<\log(|\det(B)|)$ or $A=B$. This relation is clearly a partial order. Let us verify that the map $m:GL_n(\mathbb{R})\times GL_n(\mathbb{R})\rightarrow GL_n(\mathbb{R})$ defined by $m(A,B)=AB$ is a continuous map. Suppose $(A,B)\leq (C,D)$, which implies that $\log(|\det(A)|)<\log(|\det(C)|)$ or $A=C$ and $\log(|\det(B)|)<\log(|\det(D)|)$ or $B=D$. On the other hand, $\log(|\det(AB)|)=\log(|\det(A)|)+\log(|\det(B)|)$ and $\log(|\det(CD)|)=\log(|\det(C)|)+\log(|\det(D)|)$. By studying all possible cases we can conclude that $m(A,B)\leq m(C,D)$. \begin{enumerate}
        \item[(i)] If $A=C$ and $B=D$, then $m(A,B)=m(C,D)$
        \item[(ii)] If $A=C$ and $\log(|\det(B)|)<\log(|\det(D)|)$, then $m(A,B)\leq m(C,D)$ because \begin{align*}
            &\log(|\det(AB)|)=\log(|\det(A)|)+\log(|\det(B)|)<\log(|\det(A)|)+\log(|\det(D)|)=\\&=\log(|\det(C)|)+\log(|\det(D)|)=\log(|\det(CD)|)
        \end{align*}
        \item[(iii)] If $B=D$ and $\log(|\det(A)|)<\log(|\det(C)|)$, by repeating the same argument in (ii) we obtain that $m(A,B)\leq m(C,D)$.
    \end{enumerate}
    
    This shows that $GL_n(\mathbb{R})$ with the partial order defined above is a connected, non-compact, non-abelian Alexandroff paratopological group. Note also that the elements of $SL_n(\mathbb{R})=\{A\in GL_n(\mathbb{R})~|\det(A)=1\}$ form an antichain containing the identity. Moreover, every subgroup of $GL_n(\mathbb{R})$  can be seen as an Alexandroff paratopological group with the topological structure inherited from the partial order defined on $GL_n(\mathbb{R})$.

\end{ex}

\section{Topological properties}\label{sec_topological}

\subsection{Minimal open and closed sets of the identity element}

This subsection focuses on fundamental properties of the open and closed sets around the identity element in a non‑compact Alexandroff paratopological group. We also show that studying these sets yields structural information about the entire space.

\begin{thm}\label{prop_no_hay_orden_finito_entornos_neutro}
    Let $X$ be a non-compact Alexandroff paratopological group. Then there is no element in $(U_1\cup F_1)\setminus \{ 1\}$ with finite order.
\end{thm}
\begin{proof}
    Suppose there exists $x\in F_1\setminus \{1\}$ such that $x\neq 1$ and $x^n=1$ for some positive integer number $n$. By applying $L_x$ successively, we obtain $$x^n=1>x^{n-1}>\cdots>x^2>x>1,$$
    which entails a contradiction. The same argument also applies when $x\in U_1\setminus \{ 1\}$.
\end{proof}

\begin{rem} 

Theorem \ref{prop_no_hay_orden_finito_entornos_neutro} shows that there are no elements of finite order in $(U_1\cup F_1)\setminus \{1\}$ for any non-compact Alexandroff paratopological group. However, it is possible to have elements of finite order. For example, consider $GL_n(\mathbb{R})$ from Example \ref{ex_matrices_invertibles}. It is evident that for every permutation matrix $P$, one has that $P^n$ is the identity matrix.  This does not contradict Theorem \ref{prop_no_hay_orden_finito_entornos_neutro} because $P$ is not comparable to the identity matrix.
    
\end{rem}

As an immediate consequence of Theorem \ref{prop_no_hay_orden_finito_entornos_neutro}, we deduce the following result: 
\begin{cor}
    Let $X$ be a non-compact Alexandroff paratopological group. If $x\in X$ has finite order, then $\{1,x \}$ is an antichain in $X$.
\end{cor}

\begin{thm}\label{thm_strong_local_homogeneuos}
    Let $X$ be a non-compact Alexandroff paratopological group. For any pair of points $x,y\in X$, $U_x$ and $F_x$ are homeomorphic to $U_1$ and $F_1$, respectively.
\end{thm}
\begin{proof}
    By hypothesis, the map $L_x:X\rightarrow X$ is a homeomorphism. Clearly, $L_x(U_1)$ (resp., $L_x(F_1)$) is homeomorphic to $U_x$ (resp., $F_x$). From this, we conclude the result. 
\end{proof}

\begin{cor}\label{lem_auxiliar_probar_hyperconexos}
    Let $X$ be a non-compact Alexandroff paratopological group. If $x,y,z\in X$ are such that $z < x,y$ (resp., $x,y < z$), then there exists $\overline{z}\in X$ such that $\overline{z} > x,y$ (resp.,\ $\overline{z} < x,y$).
\end{cor}
\begin{proof}
    We prove the first case, since the other one is analogous. The condition $x,y > z$ is equivalent to $U_x \cap U_y \neq \emptyset$. By Theorem~\ref{thm_strong_local_homogeneuos} and Lemma~\ref{lem_opuesto_u_1}, we have that $U_x^{\mathrm{op}} = F_x$ and $U_y^{\mathrm{op}} = F_y$. Therefore, $F_x \cap F_y \neq \emptyset$, which ensures the existence of $\overline{z}$.
\end{proof}

Theorem \ref{thm_strong_local_homogeneuos} tells that every non-compact Alexandroff paratopological group satisfies a stronger version of the definition of being locally homogeneous. Moreover, from this result, it is immediate the proof of the following result.

\begin{cor}\label{cor_weak_beat_si_1_lo_es}
    Let $X$ be a non-compact Alexandroff paratopological group and $x\in X$. Then $x$ is a beat point if and only if $1$ is a beat point. Particularly, every point of $X$ is a beat point if and only if $1$ is a beat point.
\end{cor}

\begin{ex}\label{ex_z_z}
    Let us consider $\mathbb{Z}\oplus \mathbb{Z}$ with $(a,b)\leq (c,d)$ if and only if $a\leq c$ and $b\leq d$, and $(a,b)+(c,d)=(a+c,b+d)$, where $+$ stands for the usual sum. It is easy to verify that $\mathbb{Z}\oplus \mathbb{Z}$ is a non-compact Alexandroff paratopological group. We have depicted in Figure \ref{fig:diagram_hasse_z_z}  its Hasse diagram and also the open set $U_{(0,0)}$ in red and the closed set $F_{(0,0)}$ in blue. From Proposition \ref{cor_weak_beat_si_1_lo_es}, we deduce that $\mathbb{Z}\oplus \mathbb{Z}$ does not have beat points.
    \begin{figure}[ht]
        \centering
        \includegraphics[width=0.8\linewidth]{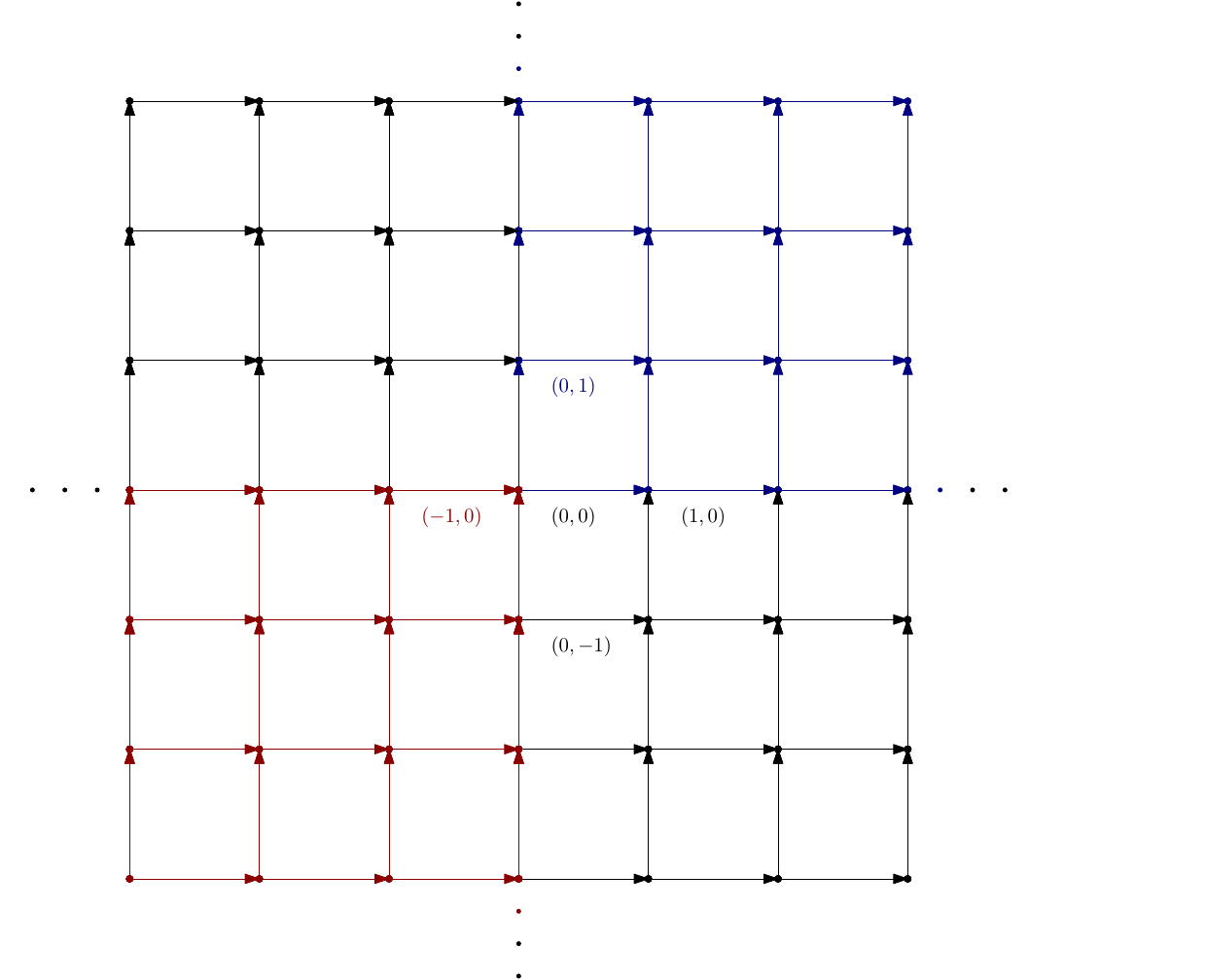}
        \caption{Hasse diagram of $\mathbb{Z}\oplus \mathbb{Z}$ from Example \ref{ex_z_z}.}
        \label{fig:diagram_hasse_z_z}
    \end{figure}
\end{ex}

\begin{thm}
    Let $X$ be a non‑compact Alexandroff paratopological group. Then a point $x\in X$ is a beat point if and only if $X$ is isomorphic to the Alexandroff paratopological group $(\mathbb{Z},\leq)$. In particular, $X$ is contractible.
\end{thm}
\begin{proof}
    We only prove the nontrivial implication. Without loss of generality, we assume that $x$ is an up beat point. By Corollary~\ref{cor_weak_beat_si_1_lo_es}, every point $x \in X$ is an up beat point. We now show that $F_1$ is a chain. By hypothesis, $x_1 \succ 1$, where $x_1$ is the minimum of $F_1 \setminus \{1\}$. Repeating the same argument, we obtain $x_2 \succ x_1 \succ 1$, where $x_2$ is the minimum of $F_{x_1} \setminus \{x_1\}$. Continuing inductively, we see that $F_1$ consists of the chain
\[
    \cdots \succ x_n \succ x_{n-1} \succ \cdots \succ x_2 \succ x_1 \succ 1.
\]

Finally, Lemma~\ref{lem_opuesto_u_1} implies that $X$ is homeomorphic to $(\mathbb{Z},\leq)$, since $X$ is given by the chain
\[
    \cdots \succ x_2 \succ x_1 \succ 1 \succ x_1^{-1} \succ x_2^{-1} \succ \cdots.
\]
Moreover, $x_1$ and $x_1^{-1}$ clearly generate $X$. Hence, $X$ is isomorphic to $(\mathbb{Z},\leq)$.
\end{proof}

\textbf{Open problem.} A weaker notion than that of a beat point is the concept of a weak beat point (see \cite{barmak2011algebraic}). This raises the following question: \emph{can we characterize non‑compact Alexandroff paratopological groups with a weak beat point?}

\subsection{Connectedness properties}

The aim of this subsection is to prove that every non-compact Alexandroff paratopological group is hyperconnected. As an immediate consequence, we obtain several properties of its open and closed subsets.

\begin{df}
    A topological space $X$ is said to be \emph{hyperconnected} (resp., \emph{ultraconnected}) if the intersection of every pair of non-empty open sets (resp., closed sets) is non-empty. Equivalently, $X$ has no two disjoint non-empty open subsets (resp., closed subsets).
\end{df}
\begin{rem}\label{rem_product_hyperconnected}
    Let us recall that the arbitrary product of hyperconnected (resp., ultraconnected) topological spaces is hyperconnected (resp., ultraconnected).
\end{rem}
\begin{prop}
    Let X be a non-compact Alexandroff paratopological group. Then X is hyperconnected if and only if X is ultraconnected.
\end{prop}

\begin{proof}
Suppose that X is hyperconnected and let $x, y \in X$. By hypothesis, $U_x \cap U_y\neq \emptyset$ and,
by Corollary \ref{lem_auxiliar_probar_hyperconexos}, $F_x \cap F_y\neq \emptyset$.

The proof of the other implication is identical.
\end{proof}

\begin{prop}
    Let $X$ be a non-compact Alexandroff paratopological group. Then $X$ is a directed set if and only if $X$ is hyperconnected.
\end{prop}
\begin{proof}
    Suppose $X$ is a directed set. Let $x,y\in X$. By hypothesis, there exists $z\in X$ such that $x,y<z$. Therefore, $F_x\cap F_y\neq \emptyset$. By Corollary \ref{lem_auxiliar_probar_hyperconexos}, we deduce that $U_x\cap U_y\neq \emptyset$.

    Suppose now that $X$ is hyperconnected and let $x,y\in X$. By hypothesis, $U_x\cap U_y\neq \emptyset $ and, consequently, there exists $z<x,y$. From this, we can conclude the desired result.
\end{proof}


\begin{thm}\label{every_paratopological_alexandroff_hyperconexosiiiiii}
   Every connected non-compact Alexandroff paratopological group $X$ is hyperconnected.
\end{thm}
\begin{proof}
Let $x,y \in X$. By the connectedness of $X$, there exists a minimal sequence of points $x_0,\ldots,x_n \in X$ such that $x_0 = x$, $x_n = y$, and for each $i = 0,\ldots,n-1$ we have either $x_i < x_{i+1}$ or $x_i > x_{i+1}$.  

The minimality of the sequence guarantees that for each $i$,
\[
x_i < x_{i+1} > x_{i+2} \quad \text{or} \quad x_i > x_{i+1} < x_{i+2},
\]
since otherwise we could delete the middle point, contradicting minimality.

Without loss of generality, assume that
\[
x < x_1 > x_2 < x_3 > \cdots > y.
\]

By the second part of Corollary~\ref{lem_auxiliar_probar_hyperconexos}, there exists $\overline{x}_1$ such that  
\[
x > \overline{x}_1 < x_2.
\]
Replace $x_1$ by $\overline{x}_1$ in the sequence. Then we obtain
\[
x > \overline{x}_1 < x_2 < x_3 > \cdots>y,
\]
which contradicts the minimality of the sequence unless $x_2 = y$.

Therefore, for any two points $x,y \in X$ we have either $U_x \cap U_y \neq \emptyset$ or $F_x \cap F_y \neq \emptyset$, which proves the result.
\end{proof}

\begin{cor}
Let $X$ be a connected non-compact Alexandroff paratopological group and let $A\subset X$.
\begin{enumerate}
    \item If $A$ is a non-empty open set, then $A$ is dense.
    \item If $A$ is a proper closed set, then $A$ is nowhere dense.
\end{enumerate}
\end{cor}

\subsection{The product of feebly bounded sets is feebly bounded}

This section is devoted to introducing the fundamental terminology required for tackling the problems stated in \cite[Open Problem 6.10.1 and 6.10.2]{TArhangelskiiTkachenko2008} within our setting.

\begin{df}
   A space $X$ is called \emph{feebly compact} if every locally finite family of open sets in $X$ is finite. Recall that a family $\mathcal{U}$ of open sets is \emph{locally finite} whenever, for every $x \in X$, there exists an open neighborhood of $x$ that intersects only finitely many members of $\mathcal{U}$.
\end{df}


\begin{thm}\label{every_alexandroff_feebly_compact}
    Let $X$ be a connected non-compact Alexandroff paratopological group. Then $X$ is feebly compact.
\end{thm}
\begin{proof}
Suppose that $\mathcal{U}$ is a locally finite family of open subsets of $X$. Given $x\in X$, choose an open neighborhood $U$ of $x$. Since $X$ is hyperconnected (Theorem \ref{every_paratopological_alexandroff_hyperconexosiiiiii}), $U$ intersects every nonempty member of $\mathcal{U}$. It follows that $\mathcal{U}$ is finite.
\end{proof}



\begin{df}
    A subset $B$ of a space $X$ is \emph{feebly bounded} in $X$ if every locally finite family of open sets in $X$ contains at most finitely elements that intersect $B$.
\end{df}

\begin{thm}\label{cor_feebly_bounded_siempre}
    Let $X$ be a connected non-compact Alexandroff paratopological group. Then every subset $B$ of $X$ is feebly bounded in $X$. Particularly, $X$ is feebly bounded.
\end{thm}
\begin{proof}
By Theorem \ref{every_alexandroff_feebly_compact}, $X$ is feebly compact. This gives that every locally finite family of open sets $\mathcal{U}$ is finite and, consequently, $\mathcal{U}$ contains at most finitely elements that interesect $B$.
\end{proof}

We now have all the necessary ingredients to answer, in our setting, the questions raised in \cite[Open Problems 6.10.1 and 6.10.2]{TArhangelskiiTkachenko2008}. In fact, our approach yields a more general statement.

\begin{cor}   Let $\{X_\alpha\}_{\alpha\in \Lambda}$ be a family of connected non-compact Alexandroff paratopological groups. If $\{B_\alpha \}_{\alpha\in \Lambda}$  is a family such that $B_\alpha\subset X_\alpha$ for each $\alpha\in \Lambda$, then $\prod_{\alpha\in \Lambda} B_\alpha$ is a feebly bounded subset of $\prod_{\alpha\in \Lambda}X_\alpha$.  
\end{cor}

\begin{cor}\label{problema_2_abierto_tachenko}
Let $X$ be a connected non-compact Alexandroff paratopological group. Then $AB$ is a feebly bounded subset of $X$ for all subsets $A,B$ of $X$. 
\end{cor}

\section{The combinatorics of Alexandroff paratopological groups}\label{sec_combinatorics}
We adapt and investigate several well‑known properties of partially ordered sets within this framework. This approach leads to a classification of a family of non‑compact Alexandroff paratopological groups.

\begin{prop}\label{cor_ht_no_acotado}
    Let $X$ be a non-compact Alexandroff paratopological group. Then $\textnormal{ht}(X)=\infty$. 
\end{prop}
\begin{proof}
    We argue by contradiction. Suppose $\textnormal{ht}(X)=n$, which gives that there exists a chain $x_0<x_1<...<x_n$. On the other hand, as an immediate consequence of Theorem \ref{thm_strong_local_homogeneuos}: $F_{x_0}$ is homeomorphic to $F_{x_n}$. Thus, we can extend the chain $x_0<x_1<...<x_n$ with elements from $F_{x_n}$ and this gives the contradiction.
\end{proof}

Unlike with the height, we may have finite width as we demonstrate with the following non-trivial example.

\begin{ex}\label{ex_ancho_2}
    Let $X=\{x_i,y_j \}_{i,j\in \mathbb{Z}}$ with $x_i,y_j< x_s,y_k$ if and only if $i,j<k,s$. It is evident that $X$ is a partially ordered set with this relation (see Figure \ref{fig_ancho_2} for the Hasse diagram of $X$). For every $i\in \mathbb{Z}$, define $x_i:X\rightarrow X$ by $x_i(x_j)=x_{j+i}$ and $x_i(y_j)=y_{j+i}$, and define $y_i:X\rightarrow X$ by $y_i(x_j)=y_{j+i}$ and $y_{i}(y_j)=x_{j+i}$. Let us verify that $x_i$ and $y_i$ are continuous maps. We first prove that $x_i$ is continuous by studying cases. 
    \begin{itemize}
    \item[(i)] $x_j<x_s$, then $j<s$ and we obtain $x_i(x_j)=x_{j+i}<x_{s+i}=x_i(x_s)$. 
    \item[(ii)] $x_j<y_k$, then $j<k$ and, therefore, $x_i(x_j)=x_{j+i}<y_{k+i}=x_i(y_k)$ since $j+i<k+i$. 
    \item[(iii)] $y_{j}<x_s$ implies $j<s$ and we deduce $x_i(y_j)=y_{j+i}<x_{s+i}=x_i(x_s)$. 
    \item[(iv)] $y_j<y_k$ gives $j<k$, then $x_i(y_j)=y_{j+i}<y_{k+i}=x_i(y_k)$. 
    \end{itemize}
    Let us prove that $y_j$ is also continuous.
    \begin{itemize}
        \item[(i)] $x_i<y_k$, then $i<k$ and we deduce $y_j(x_i)=y_{i+j}<x_{k+j}=y_j(y_k)$. 
        \item[(ii)] $x_i<x_s$ means $i<s$ and, consequently, $y_j(x_i)=y_{i+j}<y_{s+j}=y_j(x_s)$. 
        \item[(iii)] $y_i<y_k$ gives $i<k$ and we obtain $y_j(y_i)=x_{i+j}<x_{k+j}=y_j(y_k)$. \item[(iv)] $y_i<x_s$ implies $i<s$ and we deduce $y_j(y_i)=x_{i+j}<y_{s+j}=y_j(s_s)$.
        \end{itemize}
        
        The maps $x_i$ and $y_j$ are indeed homeomorphism. This provides the structure of group in $X$, where we consider the composition of maps as the group operation. It is simple to deduce that the map $m:X\times X\rightarrow X$ given by $m(a,c)=a\cdot c$ is continuous. Thus, $X$ is a non-compact Alexandroff paratopological group.
 
    Moreover, $X$ satisfies $\textnormal{width}(X)=2$. Note that $U_{x_i}\cup U_{y_i}$ (resp., $F_{x_i}\cup F_{y_i}$) is a finite model of the infinite dimensional sphere for every $i\in \mathbb{Z}$. Additionally, it is not difficult to verify that  $X$ is contractible.

    \begin{figure}[ht]
        \centering
        \includegraphics[width=0.25\linewidth]{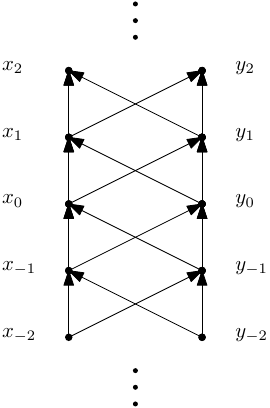}
        \caption{Hasse diagram of X from Example \ref{ex_ancho_2}.}
        \label{fig_ancho_2}
    \end{figure}
\end{ex}

\begin{df}\label{df_radius}
    Let $X$ be a non-compact Alexandroff paratopological group. We define the upper ball of minimal radius of $x\in X$, denoted by $r(x)$, as $r(x)=\{y\in X~|~ x\prec y \}$. Moreover, we define the radius of $X$ to be $|r(1)|$, and we denote it by $r(X)$.
\end{df}
Clearly, the radius is a topological invariant and $|r(x)|=|r(1)|$ for every $x\in X$ by Theorem \ref{thm_strong_local_homogeneuos}. However, this invariant is far from classifying the family of Alexandroff paratopological groups. In Example \ref{ex_ancho_2}, we conclude that the radius is 2. On the other hand, in Example \ref{ex_reales_orden_usual}, the radius is equal to zero since there is no point $x\in \mathbb{R}$ such that $0\prec x$.  By the definition of radius, the following two lemmas follow easily:
\begin{lem}\label{lem_radius_menor_width}
    Let $X$ be a non-compact Alexandroff paratopological group. Then $r(X)\leq \textnormal{width}(X)$.
\end{lem}

\begin{lem}
    Let $X$ be a non-compact Alexandroff paratopological group. If $r(X)=1$, then $X$ is a disjoint union of $\textnormal{width}(X)$ non-finite chains. Particularly, if $X$ is connected, then $X$ is a non-finite chain.
\end{lem}

In general, the notion of width need not to be equal to the notion of radius. For instance, see $\mathbb{Z}\oplus\mathbb{Z}$ from Example \ref{ex_z_z}, $\textnormal{width}(\mathbb{Z}\oplus \mathbb{Z})=\infty$ and $r(\mathbb{Z}\oplus \mathbb{Z})=2$. This also happens even if the width is finite as we demonstrate with the following example.

\begin{ex}\label{ex_generalizacion_join}
    Let $\mathbb{Q} \oplus \mathbb{Z}_n$. Consider the partial order defined by
\[
(a,b) \le (c,d) \quad \text{if and only if} \quad a \le c \ \text{and}\ b = d.
\]
With this relation, the multiplication map $m$ from Definition~\ref{def_topological group} 
is continuous, and therefore $\mathbb{Q} \oplus \mathbb{Z}_n$ is a non-connected non-compact 
Alexandroff paratopological group. Moreover, $\operatorname{width}(\mathbb{Q} \oplus \mathbb{Z}_n) = n$
and $r(\mathbb{Q} \oplus \mathbb{Z}_n) = 0$. By replacing $\mathbb{Q}$ with $\mathbb{Z}$, 
we obtain $\operatorname{width}(\mathbb{Z} \oplus \mathbb{Z}_n) = n$ and 
$r(\mathbb{Z} \oplus \mathbb{Z}_n) = 1$. Note that $\mathbb{Z} \oplus \mathbb{Z}_n$ is neither connected nor hyperconnected.
\end{ex}

\begin{rem}\label{rem_generalizacion_join}
Example \ref{ex_ancho_2} can be easily generalized to width $n<\infty$ by considering the non-Hausdorff join of antichains with $n$ points $\{ x_1,...,x_n\}$. This can also be seen from $\mathbb{Z}\oplus \mathbb{Z}_n$ in Example \ref{ex_generalizacion_join}. We just need to consider the partial order given as follows:
\[
(a,b) \le (c,d) \quad \text{if and only if} \quad a \le c \ \text{and}\ b = d \ \text{or } a<c.
\]

For that case, a deep study about the properties of $\bigcup_{i=1}^n U_{x_i}$ has been carried out in \cite{chocano2025riordan}. This also proves the non-triviality and the interest of this family of spaces.
\end{rem}

\begin{prop}
    Let $X$ and $Y$ be two non-compact Alexandroff paratopological groups. If $r(X),r(Y)<\infty$, then $r(X\times Y)\leq 2r(X)r(Y)$.
\end{prop}
\begin{proof}
    Let $r(X)=n$, $r(Y)=m$. Then there are only  $1\prec x_i$ with $i=1,...,n$ elements in $X$. Similarly, there are only $1\prec y_j$ with $j=1,...,m$ elements in $Y$. Clearly, $(1,1)\prec (1,y_i),(x_i,1)$.
\end{proof}

\begin{lem}\label{lem_ancho_anticadena_1}
    The width of a non-compact Alexandroff paratopological group $X$ is provided by the number of elements in a maximal antichain containing $1$.
\end{lem}
\begin{proof}
    Let $A_1$ denote a maximal antichain containing $1$. Then, it is clear that $|A_1|\leq \textnormal{width}(X)$. Suppose that there exists an antichain $A$ that does not contain $1$ such that $|A|=\textnormal{width}(X)$. Let $x\in A$. Since $L_{x^{-1}}$ is a homeomorphism, it preserves antichains. Therefore, $L_{x^{-1}}(A)$ is an antichain containing $1$ which yields that $|A|\leq |A_1|$ by the maximality of $A_1$.
\end{proof}

\begin{rem}\label{rem_recubrimiento_anticadenas_c}
    For any Alexandroff paratopological group $X$, one has that $X=\bigcup_{x\in A_1}C_x$.
\end{rem}


We now proceed to classify the non‑compact Alexandroff paratopological groups for which the radius equals the width. 

\begin{thm}\label{thm_ancho_finito_clasificacion_join}
    Let $X$ be a connected non-compact Alexandroff paratopological group. If $r(X)=\textnormal{width}(X)<\infty$  then $X$ is homeomorphic to $$\cdots\circledast Y_n\circledast Y_n\circledast Y_n \circledast\cdots, $$
    where $Y_{n}$ is an antichain of $n$ elements.
\end{thm}
\begin{proof}
    Let  $r(1)=\{x_1,...,x_n \}$. First, we prove that for every $y\in r(x_i)$, one has that $y>x_j$ for every $1\leq j\leq n$. The map $L_{x_i}$ is a homeomorphism for every $i$. Then, it preserves antichains and relations $\prec$, consequently, $L_{x_i}:r(1)\rightarrow r(x_i)$ is a homeomorphism for every $i$. By cardinality, we obtain that $L_{x_i}(r(1))=L_{x_j}(r(1))$. Hence, $y>x_j$ for every $y\in r(x_i)$. 

    From this, it can be deduced that the space $X$ is indeed homeomorphic to $$\cdots\circledast Y_n\circledast Y_n\circledast Y_n \circledast\cdots, $$
    where $Y_n$ is an antichain of $n$ points.

\end{proof}

As an immediate consequence of the structure provided in Theorem \ref{thm_ancho_finito_clasificacion_join}, we obtain the following result.

\begin{cor}
    Let $X$ be a connected non-compact Alexandroff paratopological group. If $r(X)=\textnormal{width}(X)<\infty$, then $X$ is contractible.
\end{cor}
\begin{proof}
    Immediate consequence of \cite{chocano2025riordan}, where the authors prove that $Y_{n}\circledast Y_n\circledast\cdots$ is contractible.
\end{proof}

\bibliography{bibliografia}
\bibliographystyle{plain}

\newcommand{\Addresses}{{
  \bigskip
  \footnotesize

    \textsc{T. Borsich, Departamento de Matemática Aplicada, Ciencia e Ingeniería de los Materiales y Tecnología Electrónica, ESCET Universidad Rey Juan Carlos, 28933 Móstoles (Madrid), Spain}\par\nopagebreak
 \textit{E-mail address}: \texttt{tayomara.gonzalez@urjc.es}

  \textsc{ P.J. Chocano, Departamento de Matemática Aplicada, Ciencia e Ingeniería de los Materiales y Tecnología Electrónica, ESCET Universidad Rey Juan Carlos, 28933 Móstoles (Madrid), Spain}\par\nopagebreak
 \textit{E-mail address}: \texttt{pedro.chocano@urjc.es}

}}

\Addresses

\end{document}